\documentclass{article}

\usepackage{amsfonts,amsmath,amsthm,amssymb,enumerate,mathtools,relsize,graphicx,bm}
\usepackage{multicol}
\usepackage{hyperref}
\usepackage{xcolor}
\usepackage[inline]{enumitem}
\usepackage{cases}
\usepackage{tikz}
\usepackage{dsfont}

\theoremstyle{plain}
\newtheorem{theorem}{Theorem}[section]
\newtheorem{corollary}[theorem]{Corollary}
\newtheorem{proposition}[theorem]{Proposition}
\newtheorem{lemma}[theorem]{Lemma}

\newtheorem{conjecture}[theorem]{Conjecture}
\theoremstyle{definition}

\newtheorem{remark}[theorem]{Remark}

\usepackage[normalem]{ulem}

\newcommand{\Gap}{\mathrm{Gap}}
\newcommand{\dist}{\mathrm{dist}}
\newcommand{\diam}{\mathrm{diam}}

\title{A Nordhaus--Gaddum problem for spectral gap of a graph
}
\author{
Sooyeong Kim\footnote{Contact: kimswim@yorku.ca} \and
Neal Madras\footnote{Contact: madras@yorku.ca}
}

\date{Department of Mathematics and Statistics, York University, Toronto, Canada\\[3ex]
\today}

\begin{document}
	
	\maketitle
	
	\begin{abstract}
		Let $G$ be a graph on $n$ vertices, with complement $\overline{G}$. The spectral gap of the transition probability matrix of a random walk on $G$ is used to estimate how fast the random walk becomes stationary.  We prove that the larger spectral gap of $G$ and $\overline{G}$ is $\Omega(1/n)$. Moreover, if all degrees are $\Omega(n)$ and $n-\Omega(n)$, then the larger spectral gap of $G$ and $\overline{G}$ is $\Theta(1)$. We also show that if the maximum degree is $n-O(1)$ or if $G$ is a join of two graphs, then the spectral gap of $G$ is $\Omega(1/n)$. Finally, we provide a family of connected graphs with connected complements such that the larger 
		spectral gap of $G$ and $\overline{G}$ is $O(1/n^{3/4})$.
	\end{abstract}
	
	
	
	
	\section{Introduction}
	
	In graph theory, Nordhaus--Gaddum questions explore the relationship between a (finite, simple) graph $G$ and its complement, denoted $\overline{G}$, relative to some graph invariant. Given a graph invariant $f(\cdot)$, a Nordhaus--Gaddum question considers bounds on either the sum $f(G)+f(\overline{G})$ or the product $f(G)f(\overline{G})$, or on the minimum or maximum of the set $\{f(G), f(\overline{G})\}$. Nordhaus and Gaddum originally studied such questions for the chromatic number of a graph \cite{nordhaus1956complementary}. Since then, Nordhaus--Gaddum questions have been studied for a wide range of graph invariants; see \cite{aouchiche2013survey} for a survey of results of this kind. 
	
	In this paper, we take a Nordhaus--Gaddum viewpoint on the spectral gap of a graph, i.e.\ on the spectral gap of random walk on a graph.  (See Sections \ref{sec:notation}--\ref{sec:MCcoupling} for full definitions.)  Briefly, the random walk on a graph $G$ corresponds to a Markov chain whose transition probability matrix we denote $P_G$, and the spectral gap $\Gap(G)$ of $G$ (or of $P_G$) is the difference between the two largest eigenvalues of $P_G$ (the largest eigenvalue necessarily being 1). Roughly speaking, the spectral gap  
	corresponds to the exponential rate at which the distribution of the random walk $P_G$ approaches its equilibrium distribution (ignoring periodicity issues, which often can be handled by other means \cite{rosenthal2003}). The reciprocal of the spectral gap is called the relaxation time, and indicates the time scale needed for equilibration.
	Work in~\cite{chung1997spectral} provides a lower bound on the spectral gap based on the graph's diameter and number of edges. In \cite{aksoy2018maximum}, it is shown that the maximum relaxation time over all connected graphs on $n$ vertices is asymptotically $n^3/54$ and can be attained by the so-called \textit{barbell graph}. At the other extreme, the complete graph attains the exact minimum relaxation time of $(n-1)/n$ 
	\cite[Section~5]{aldous2002reversible}.

	Intuitively, if $G$ has an unusually small spectral gap, then the random walk on $G$ has a hard time moving around $G$, due to $G$ not being well connected in some way.  In that case, it is reasonable to guess that the complement graph $\overline{G}$ has the connections that $G$ is missing, and that the spectral gap of $\overline{G}$ is not small.  Questions of Nordhaus--Gaddum type try to make sense of this vague reasoning.  For example, in the extreme case, suppose that the $n$-vertex graph $G$ is disconnected, with two components $C_1$ and $C_2$ (with $n_1$ and $n_2$ vertices, respectively).  Then $\Gap(G)=0$.  
	However, in $\overline{G}$, every vertex of $C_1$ is adjacent to every vertex of $C_2$, and this is enough to show that $\Gap(\overline{G})$ is not small.  Indeed, Theorem \ref{thm.joingap} implies that in this case we have 
	$\Gap(\overline{G})\geq \min\{n_1,n_2\}/32n$. 
	One interpretation of this paper's results is that a spectral gap of $o(1/n)$ is sufficient to be ``unusually small'' for an $n$-vertex graph.
	
	The work of \cite{faught2023nordhaus} studies the Nordhaus--Gaddum problem for the second largest eigenvalue of the normalized Laplacian matrix, which is the same as the spectral gap. It shows that $\max\{\Gap(G),\Gap(\overline{G})\} \geq2/n^2$, using arguments based on isoperimetry and Cheeger's constant,
	and makes the following two conjectures. 
	
	\begin{conjecture}\label{conjecture1}\cite{faught2023nordhaus}
		Let $G$ be a graph on $n$ vertices. Then $$\max\{\Gap(G), \Gap(\overline{G})\} \ge \frac{2}{n-1}.$$
	\end{conjecture}
	\begin{conjecture}\cite{faught2023nordhaus}\label{conjecture2}
		Let $G$ be a graph on $n$ vertices. If $G$ and $\overline{G}$ are both connected, then $$\Gap(G)+ \Gap(\overline{G}) \ge \frac{2}{\sqrt{n}}.$$
	\end{conjecture}
	
	\noindent
	As observed in \cite{faught2023nordhaus}, the family of graphs described in our Remark~\ref{rem:diam2} attains the lower bound of Conjecture \ref{conjecture1}. It happens that the complement of each graph in this family is disconnected.  Conjecture \ref{conjecture2} arose from numerical observations suggesting that a better bound held if a graph and its complement were both connected. Among the present paper's results, we shall prove the first conjecture up to a constant factor, and disprove the second conjecture.  
	
	
	Here is what we shall prove instead of Conjecture \ref{conjecture1}. 
	\begin{theorem}\label{thm:main1}
		Let $G$ be a graph on $n$ vertices for any $n\geq 2$. Then $$\max\{\Gap(G),\Gap(\overline{G})\} \;\geq\; \frac{3-\sqrt{5}}{8}\frac{1}{n-1} \;=\; \frac{0.09549\ldots}{n-1} \,.  $$
	\end{theorem}
	The proof uses a probabilistic  
	construction together with graph theoretical arguments based on diameter, along the following lines.  It is well known that if the diameter of $G$ is more than 3, then the diameter of $\overline{G}$ must be 2.  We use a coupling argument to prove that $\Gap(H)\geq \frac{1}{n-1}$ if $H$ has diameter 2 (we actually prove something a bit better).  Allowing either $G$ or $\overline{G}$ to assume the role of $H$, we see that the only remaining case to consider is when both $G$ and $\overline{G}$ have diameter equal to 3, and that is the case where the constant $\frac{3-\sqrt{5}}{8}$ arises in our proof.
	
	We can also exhibit large families of graphs in which $\Gap(G)+\Gap(\overline{G})$ is bounded away from 0 by a constant.  This happens when the minimum degree is at least a constant times $n$ in both $G$ and in $\overline{G}$.  
	Specifically, we prove the following.  
	\begin{theorem} 
		\label{thm.complementLU}
		Let $L$ and $U$ be real constants such that $0<L< U<1$.  Then 
		for every $n\geq \,\frac{2}{1-U}$, and for every
		graph $G$ on $n$ vertices having the property that each vertex degree is in $[Ln,Un]$, we have
		\[    \max\left\{  \Gap(G), \, \Gap(\overline{G})\right\}   \;\geq \;  \frac{L^4(1-U)^4}{2^{13}}.
		\]
	\end{theorem}
	We do not know whether the order $\frac{1}{n}$ of the lower bound in Theorem \ref{thm:main1} and Conjecture \ref{conjecture1} is optimal when $G$ and $\overline{G}$ both are connected.  However, we provide an infinite family of graphs $G$ (see Figure~\ref{fig:graph G} in Section~\ref{sec6}) such that 
	$$\max\{\Gap(G),\Gap(\overline{G})\} \leq \frac{8}{n^{3/4}} \hspace{5mm}\hbox{whenever $G$ and $\overline{G}$ are both connected, where $n=|V(G)|$}.
	$$
	Since $\Gap(G)+ \Gap(\overline{G}) \leq {16}/{n^{\frac{3}{4}}}$, this disproves Conjecture~\ref{conjecture2}. It remains an open problem to identify the smallest exponent $p$ for which there exists a constant $c_p$ such that 
	\[     \Gap(G)+\Gap(\overline{G}) \;\geq \; \frac{c_p}{n^p} \hspace{5mm}\hbox{
		whenever $G$ and $\overline{G}$ are both connected, where $n=|V(G)|$}.
	\]
	Our results show that the smallest $p$ lies in the interval $\left[\frac{3}{4},1\right]$.
	

	Our main probabilistic tool is the \textit{coupling} of Markov chains, which is frequently used to obtain quantitative estimates for mixing times (see \cite{aldous2002reversible,LPW}). The key insight of this method is that the probability that a Markov chain has ``reached stationarity'' by a certain time is bounded above by the probability that two ``correlated'' copies of the Markov chain fail to meet each other; and the quality of the upper bound depends on how the correlation is constructed. In other words, the faster they meet, the more rapidly we can prove that the Markov chain becomes stationary. See Section~\ref{sec:MCcoupling} for details.
	
	The related paper \cite{kim2023bounds} discusses analogous Nordhaus-Gaddum problems for Kemeny's constant $\mathcal{K}(G)$ of a graph.  Kemeny's constant is a different way of measuring how rapidly a random walk traverses a graph.  It is defined by $\mathcal{K}(G)=\sum_{j\in V} m_{ij} \pi(j)$, where $m_{ij}$ is the expected time to arrive at $j$ for the first time for a random walk that starts at vertex $i$, and $\pi(j)$ is the equilibrium probability of vertex $j$.  Remarkably, this sum is the same for every $i$.  As observed in \cite{kim2023bounds}, Kemeny's constant satisfies
	\begin{equation}
	\label{eq.kemenygap}
	\frac{1}{\Gap(G)} \;\leq \; \mathcal{K}(G) \;\leq \;   \frac{n}{\Gap(G)} \hspace{5mm}\hbox{for every graph $G$ with $n$ vertices.}
	\end{equation}
	In particular, our Theorem \ref{thm.complementLU} implies that $\min\{\mathcal{K}(G),\mathcal{K}(\overline{G})\} \,\leq \;n\,L^4(1-U)^4/2^{13}$, which is essentially Proposition 3.5 in \cite{kim2023bounds}. However, other direct translations of results between $\Gap(G)$ and $\mathcal{K}(G)$ do not necessarily give bounds that are as good as have been proved directly; see for example Remark \ref{remark:join}.  
	
	The paper is organized as follows. Section~\ref{sec:notation} provides standard notation and background for graphs, random walks, and asymptotics.  Section~\ref{sec:MCcoupling} summarizes basic terminology and properties of Markov chains, spectral gap, and coupling.  Section~\ref{sec2} constructs certain couplings based on common neighbours to obtain results in the sections that follow. Section~\ref{sec3} shows that if the maximum degree of $G$ is $n-O(1)$ or $G$ is a join of two graphs, then 
	the gap of $G$ is at least proportional to $1/n$. 
	Section~\ref{sec4} proves Theorem \ref{thm:main1} and its accompanying results.
	Section~\ref{sec5} proves Theorem \ref{thm.complementLU}.
	Finally, Section~\ref{sec6} presents a family of graphs such that the sum of the spectral gap for $G$ and $\overline{G}$ is $O(1/n^{\frac{3}{4}})$.
	
	All results in this article can be equivalently written in the context of the relaxation time or the second smallest eigenvalue of the normalized Laplacian matrix.

	\subsection{Notation and preliminaries}
	\label{sec:notation}
	We now introduce notation and definitions which will be used throughout the article. 
	
	All graphs in this paper are simple.
	We denote the vertex set and the edge set of 
	a graph $G$ by $V(G)\equiv V$ and $E(G)$ respectively. 
	The degree of vertex $i$ is denoted by $\deg_G(i)$. We also let $\delta(G)$ denote the minimum degree of $G$, and $\Delta(G)$ denote the maximum degree. The \emph{neighbourhood} of vertex $i$, denoted $N_G(i)$, is the set of vertices adjacent to $i$. 
	
	We write 
	$v_1-v_2-\cdots-v_t$ to denote a path of length $t-1$ from $v_1$ to $v_t$, that is, a sequence of distinct vertices $v_1,v_2,\dots,v_t$ such that $v_i$ is adjacent to $v_{i+1}$ for $i=1,\dots,t-1$. The \textit{distance} between two vertices $i$ and $j$, denoted $\dist_G(i,j)$, is defined as the length of the shortest path between vertices $i$ and $j$. Moreover, the \emph{diameter} $\diam(G)$ of $G$ is defined as $\diam(G) = \max_{i, j \in V(G)} \dist_G(i,j)$. 
	
	We let $\overline{G}$ denote the \textit{complement} of $G$, i.e.\ the simple graph with  $V(\overline{G})=V(G)$ and $E(\overline{G})=\{\{v,w\}:v\neq w,\{v,w\}\not\in E(G)\}$.

	We write the random walk on $G$ as a sequence 
	$X_0,X_1,X_2,\ldots$ of $V(G)$-valued random variables forming a Markov chain with one-step transition probabilities 
	given by $P_G(\cdot,\cdot)$ as follows:
	\begin{equation*}
	P_G(v,w)  \;=\;   \Pr(X_1\,=\,w\,|\,X_0=v)   \;=\;
	\begin{cases}    (\deg_G(v))^{-1},   & \text{if }\{v,w\}\in E(G); \\   0, & \text{otherwise} .  \end{cases}
	\end{equation*}
	The matrix $P_G=[P_G(i,j)]$ is called the \textit{transition matrix} of the random walk. Since the random walk is fully described by $P_G$ together with a probability distribution for $X_0$ (the ``initial distribution''), we shall abuse notation that $P_G$ interchangeably indicates the random walk or the transition matrix.
	For $t\in \mathbb{N}$, we shall write the $t$-step transition probabilities  
	as $P_G^{t}(v,w)\,=\,\Pr(X_t=w\,|\,X_0=v)$.  
	For $S\subseteq V$, $P_G^{t}(v,S)$ denotes $\sum_{w\in S}P_G^{t}(v,w)$.
	
	We shall also use the associated ``lazy'' random walk on $G$, whose one-step transition probabilities $\mathfrak{P}_G(v,w)$ are 
	\[
	\mathfrak{P}_G(v,w) \;=\;  \frac{1}{2}\left(\mathds{1}_{\{v=w\}} + P_G(v,w)\right)
	\]
	where $\mathds{1}_A$ is the indicator function of set $A$. Intuitively, we can think of the lazy random walk as tossing a fair coin at each step; if the coin lands Heads, the process takes a step according to $P_G$, but if it lands Tails then the process does not move.  The associated transition probability matrix is $\mathfrak{P}_G=\frac{1}{2}(I+P_G)$ where $I$ is the identity matrix.
	Assuming that $G$ is connected, the Markov chain $P_G$ (or $\mathfrak{P}_G$) has a unique stationary distribution, which 
	we shall denote $\pi(\cdot)$.  It is well known that 
	\[    \pi(v)   \;=\;  \frac{ {\deg_G(v)}}{2|E(G)|}    \hspace{5mm}\hbox{for every $v\in V(G)$.}
	\]

	\begin{remark}
		If $G$ is clear from the context, we shall omit the subscript $G$ for $N_G(\cdot)$, $\mathrm{deg}_G(\cdot)$, $\mathrm{dist}_G(\cdot,\cdot)$, $P_G(\cdot,\cdot)$, $\mathfrak{P}_G(\cdot,\cdot)$, $P_G$, and $\mathfrak{P}_G$.
	\end{remark}
	

	We use the following asymptotic notation. In this paper, 
	$n$ always denotes  
	$|V(G)|$, the order of $G$. If $G_n$ represents a graph of order $n$ in a sequence or family of graphs, and $f$ is some positive-valued graph invariant, 
	then we write $f(G_n) = O(g(n))$ if $\limsup_{n\to\infty} \frac{f(G_n)}{g(n)}$ is finite, and understand this to mean that $f(G_n)$ has at most the same order of magnitude as $g(n)$. Furthermore, we write $f(G_n) = \Omega(g(n))$ if $g(n) = O(f(G_n))$, and $f(G_n) = \Theta(g(n))$ if $f(G_n) = O(g(n))$ and $f(G_n) = \Omega(g(n))$. We write $f(G_n) = o(g(n))$ if $\lim_{n\rightarrow\infty} \frac{f(G_n)}{g(n)} = 0$, 
	and $f(G_n) \sim g(n)$ if $\lim_{n\rightarrow\infty} \frac{f(G_n)}{g(n)} = 1$. 
	We usually omit the subscript $n$ when referring to a family of graphs. 

	
	We shall also use the following elementary inequality.
	\begin{equation}
	\label{eq.rootineq}
	1-(1-a)^{1/k}   \;  \geq \;  \frac{a}{k}     \hspace{6mm}\hbox{whenever $0<a<1$ and $k$ is a positive integer.}
	\end{equation}
	This holds because $1-c^{1/k} \,=\,(1-c)/(1+c^{1/k}+c^{2/k}+\ldots +c^{(k-1)/k})  \,\geq \,(1-c)/k$, where $c=1-a\in (0,1)$.
	
	\subsection{Markov chains, spectral gap, and coupling}
	\label{sec:MCcoupling}
	In this paper we shall only work with the random walk $P_G$ and lazy walk $\mathfrak{P}_G$ for a graph $G$ with vertex set $V$, but the present subsection holds for a general Markov chain $P$ on a finite state space.

	Let the row vector $\vec{\pi}$ be the 
	stationary distribution of $P$:  that is,
	\[   \vec{\pi} \,=\, \vec{\pi}\, P    \hspace{5mm}\hbox{and}\hspace{5mm}\sum_{i\in V}\pi(i)\,=\,1.
	\]
	Let $\vec{\mu}$ be any probability distribution vector on $V$, and let $t\in \mathbb{N}$.   
	Then $\vec{\mu}\,P^t$ is the probability distribution of the corresponding Markov chain $X_0,X_1,\ldots$ at time $t$, given 
	that $X_0$ has distribution $\vec{\mu}$.
	\[    \left( \vec{\mu}\,P^t\right)(j)  \;=\; \sum_{i\in  V}\mu(i)\,P^t(i,j) \,.
	\]
	In particular, if $\vec{\mu}=\vec{\pi}$, then the distribution of $X_t$ is exactly $\vec{\pi}$ for every $t$.

	We know that if $P$ is aperiodic, then $\lim_{t\rightarrow\infty}\vec{\mu}\,P^t=\vec{\pi}$ for every such $\mu$; we want to know the speed of this convergence.
	Let
	\[   d(t)  \; :=  \;\frac{1}{2}  \sup_{\vec{\mu} \text{ prob dist}} \|  \vec{\mu}\,P^t\,-\, \vec{\pi}\|_1   
	\;=\; \frac{1}{2} \max_{i\in V} \|  \vec{\delta_{i}}\,P^t\,-\, \vec{\pi}\|_1  
	\]
	where $\|\cdot\|_1$ is the usual 1-norm and $\vec{\delta}_i$ is the unit vector with $i^{th}$ coordinate equal to 1.  (We use the factor $\frac{1}{2}$ to make $d(t)$ agree with standard probability references such as \cite{LPW}, where $d(t)$ is the maximum \textit{total variation distance} between $\vec{\mu}P^t$ and $\vec{\pi}$.) 
	
	

	Let the eigenvalues of $P$ be $1=\lambda_1>\lambda_2\geq \ldots \geq \lambda_n\geq -1$. 
	We define the spectral gap of $P$ to be $1-\lambda_2$, and denote it by $\Gap(P)$, and we write $\Gap(G)$ for $\Gap(P_G)$. It is well known (\textit{e.g.} \cite[Section~12.2]{LPW}) that 
	\begin{equation}
	\label{eq.limspecradius}
	\lim_{t\rightarrow \infty}d(t)^{1/t} \;=\,  \max\{\lambda_2,|\lambda_n|\}   \;=:\;  \lambda^*(P).
	\end{equation}
	Consider the lazy chain $\mathfrak{P} =\frac{1}{2}(I+P)$. Note that the lazy chain is aperiodic, even if $P$ is periodic.
	The eigenvalues of $\mathfrak{P}$ are $\frac{1}{2}(1+\lambda)$ where $\lambda$ ranges over eigenvalues of $P$; hence they are 
	all nonnegative, and $\lambda^*(\mathfrak{P}) \,=\, (1+\lambda_2)/2$.
	Therefore 
	\begin{equation}
	\label{eq.lazygap}   \mathrm{Gap}(\mathfrak{P})\; =\; 1-\frac{1}{2}(1+\lambda_2)  \,=\, 
	\frac{1}{2}(1-\lambda_2)\;=\; \frac{1}{2}\,\Gap(P).
	\end{equation}

	
	One very useful probabilistic technique for estimating spectral gap is \textit{coupling}, which we now introduce.
	
	Let $\vec{\rho}$ and $\vec{\tau}$ be two probability distributions on $V$.
	Let Joint$(\vec{\rho},\vec{\tau})$ be the set of all joint probability distributions $\psi$ on $V\times V$
	whose marginal distributions are $\vec{\rho}$ and $\vec{\tau}$; that is,
	\begin{eqnarray*}
		\psi(i,j) & \geq & 0   \hspace{5mm}\text{for all }(i,j)\in V\times V,
		\\
		\sum_{i,j}\psi(i,j) & = & 1 ,  
		\\
		\rho(i)   & = & \sum_j \psi(i,j)   \hspace{5mm}\text{for all }i\in V,
		\\
		\tau(j)   & = & \sum_i \psi(i,j)   \hspace{5mm}\text{for all }j\in V.
	\end{eqnarray*}
	Such a $\psi$ is called a \textit{coupling} of $\vec{\rho}$ and $\vec{\tau}$.
	The following lemma is well known (e.g.\ Proposition 4.7 in \cite{LPW}).
	
	\begin{lemma}  
		\label{lem:jointbound}
		\textit{Let $\psi\in \text{Joint}(\vec{\rho},\vec{\tau})$.  Let $(X,Y)$ be a random 2-vector with distribution $\psi$. Then}
		\[    \frac{1}{2}\|  \vec{\rho}\,-\,\vec{\tau}\|_1 \,\leq\,  \,\Pr(X\neq Y).
		\]
	\end{lemma}
	
	\begin{proof}
		We see that
		\begin{eqnarray*}
			\sum_k |\rho(k)\,-\,\tau(k)|& =& \sum_k \left|\sum_{j\neq k}\psi(k,j)-\sum_{i\neq k}\psi(i,k)\right| 
			\\
			& \leq & \sum_k \sum_{j\neq k}\psi(k,j)  \,+\, \sum_k  \sum_{i\neq k}\psi(i,k)
			\\
			& = & 2\,\Pr(X\neq Y).
		\end{eqnarray*}
		This completes the proof.
	\end{proof}
	
	To get upper bounds on $\|\vec{\mu}\,\mathfrak{P}^t \,-\, \vec{\theta}\,\frak{P}^t\|_1$, we want to look at 
	Joint$(\vec{\mu}\,\frak{P}^t \,,\, \vec{\theta}\,\mathfrak{P}^t)$.
	That is, we are interested in couplings of $X_t$ and $Y_t$, where $X_0,X_1,\ldots$ is a Markov chain for $\mathfrak{P}$
	with initial distribution $\vec{\mu}$, and $Y_0,Y_1,\ldots$ is a Markov chain for $\mathfrak{P}$
	with initial distribution $\vec{\theta}$.
	
	We shall create a $V\times V$-valued Markov chain $\{(X^{(1)}_t,X^{(2)}_t):t=0,1,2,\ldots\}$ such that 
	\\
	$\bullet$ $\{X^{(1)}_t:t=0,1,\ldots\}$ is a Markov chain on $V$ with matrix $\mathfrak{P}$ and initial distribution $\vec{\mu}$, and
	\\
	$\bullet$ $\{X^{(2)}_t:t=0,1,\ldots\}$ is a Markov chain on $V$ with matrix $\mathfrak{P}$ and initial distribution $\vec{\theta}$.
	\\
	Letting $\textbf{P}((x_1,x_2),(y_1,y_2))=\Pr\left((X^{(1)}_{1},X^{(2)}_{1})= (x_2,y_2)\;|\; (X^{(1)}_0,X^{(2)}_0) = (x_1,y_1)\right)$, 
	the above requirements are 
	achieved by constructing transition probabilities $\textbf{P}((x_1,x_2),(y_1,y_2))$  for $(x_1,x_2)$ and $(y_1,y_2)$ in $V\times V$  such that
	\begin{align}
	\label{eq.consistency}
	\begin{split}
	\sum_{y_2} \textbf{P}((x_1,x_2),(y_1,y_2))  = & \;\mathfrak{P}(x_1,y_1)   \hspace{6mm}\hbox{for all $x_1,x_2,y_1\in V$, and }
	\\
	\sum_{y_1} \textbf{P}((x_1,x_2),(y_1,y_2))  = & \;\mathfrak{P}(x_2,y_2)   \hspace{6mm}\hbox{for all $x_1,x_2,y_2\in V$}.
	\end{split}
	\end{align}
	Such a Markov chain on $V\times V$ is called a (Markovian) coupling of the chains $\{X_t^{(1)}\}$ and $\{X_t^{(2)}\}$.
	
	Once we have done this, then for each $t$ we see that the distribution of $(X_t^{(1)},X_t^{(2)})$ is in 
	Joint$(\vec{\mu}\,\mathfrak{P}^t \,,\, \vec{\theta}\,\mathfrak{P}^t)$, and hence
	\begin{equation}
	\label{eq.normleqneq}
	\frac{1}{2}\|  \vec{\mu}\,\mathfrak{P}^t \,-\, \vec{\theta}\,\mathfrak{P}^t\|_1   \;\leq \;  \,\Pr\left(X^{(1)}_t \neq X^{(2)}_t \right).
	\end{equation}
	Thus our task of getting a good lower bound on Gap$(G)$ is converted to getting a good upper bound on the above probability as  a function of $t$.
	We need some cleverness to create a suitable chain on $V\times V$ for which good bounds can be proved.
	Note, for example, that two independent chains would not be useful here, since in that case we do not even expect the right-hand
	side of Equation (\ref{eq.normleqneq}) to converge to 0 as $t\rightarrow \infty$.

	\section{Type 1 and Type 2 couplings}\label{sec2}
	Let $G$ be a graph with vertex set $V$.  For each nonnegative integer $i$, we define $D_i$ to be the set of all pairs of vertices whose distance from each other equals $i$.
	We shall deal with the random walk $P$ and the lazy walk $\mathfrak{P}$ on $G$.

	Let $D$ be a set of pairs $(w_1,w_2)$ in $V\times V$ such that $w_1$ and $w_2$ (not necessarily distinct vertices) have at least one common neighbour.  Although there is flexibility in specifying the set $D$, it is necessarily contained in  $D_0\cup D_1\cup D_2$.
	
	For $(w_1,w_2)\in D$, we define
	\[    U_{w_1,w_2}(x)  \;=\;  \begin{cases} \frac{1}{\left|N(w_1)\cap N(w_2) \right|}  & \hbox{if }x\in N(w_1)\cap N(w_2), \\  
	0 & \hbox{otherwise}. \end{cases}
	\]
	Then $U_{w_1,w_2}$ is a probability distribution on $V$. We see that
	\begin{align*}
	0 \;\leq \;  \frak{P}(w_2,x) -\frac{C_{w_1,w_2}}{2(n-1)}U_{w_1,w_2}(x) \;=:\; Q^*_{w_1}(w_2,x)\quad\quad\text{for $x\in V$}
	\end{align*}
	where $C_{w_1,w_2} = |N(w_1)\cap N(w_2)|$. Define
	$$
	Q_{w_1}(w_2,x) = \frac{2(n-1)}{2(n-1)-C_{w_1,w_2}}Q^*_{w_1}(w_2,x).
	$$
	Observe that $Q_{w_1}(w_2,\cdot)$ is a probability distribution.
	Then, for any $(w_1,w_2)\in D$, 
	$$
	\frak{P}(w_2,x) = \left(1-\frac{C_{w_1,w_2}}{2(n-1)}\right)Q_{w_1}(w_2,x)+ \frac{C_{w_1,w_2}}{2(n-1)}U_{w_1,w_2}(x).
	$$
	Given $(w_1,w_2)\in D$, we interpret this decomposition of the distribution $\frak{P}(w_2,\cdot)$ as follows.  We toss a coin that has probability $1-\frac{C_{w_1,w_2}}{2(n-1)}$ of landing Heads. If the coin lands Tails (which it does with probability $\frac{C_{w_1,w_2}}{2(n-1)}$), then we select $x$ uniformly at random from the set $N(w_1)\cap N(w_2)$. If the coin lands Heads, then we select $x$ according to the ``residual'' probability 
	distribution $Q_{w_1}(w_2,\cdot)$.  
	
	Given the transition probability matrix $\frak{P}=\frac{1}{2}(I+P)$, we shall define a Markov chain $\{(X_t^{(1)},X_t^{(2)})\}_{t\geq 0}$ on $V\times V$ with transition probability $\mathbf{P}$ as follows:
	\bigskip
	
	\begin{subequations}
		$\textbf{P}((w_1, w_2),(v_1, v_2))$\vspace{-1mm}\label{Eqn:2-coupling}
		\begin{numcases}{=}\label{Eqn:2-coupling-1}
		\frak{P}(w_1,v_1)\mathds{1}_{\{v_1 = v_2\}}, & \text{if $w_1 = w_2$;}\\\label{Eqn:2-coupling-2}
		\left(1-\frac{C_{w_1,w_2}}{2(n-1)}\right)Q_{w_2}(w_1,v_1)Q_{w_1}(w_2,v_2), & \text{if $(w_1,w_2)\in D\backslash D_0$ and $v_1\neq v_2$;}\\\label{Eqn:2-coupling-3}
		\left(1-\frac{C_{w_1,w_2}}{2(n-1)}\right)Q_{w_2}(w_1,v_1)Q_{w_1}(w_2,v_1)+\frac{C_{w_1,w_2}}{2(n-1)}U_{w_1,w_2}(v_1), & \text{if $(w_1,w_2)\in D\backslash D_0$ and $v_1= v_2$;}\\ \label{Eqn:2-coupling-4}
		\frak{P}(w_1,v_1)\frak{P}(w_2,v_2), & \text{otherwise.}
		\end{numcases}
	\end{subequations}
	It is easy to check that this is a coupling of two desired Markov chains with matrix $\frak{P}$, by verifying \eqref{eq.consistency}.
	We call the Markov chain the \textit{Type 1 coupling for $G$ with respect to $D$}. 
	
	We can think of the mechanism of the chain $\vec{X}_t = (X_t^{(1)},X_t^{(2)})$ as follows.  If the two components are equal at some time $T$, then they remain equal for all times $t\geq T$ (by Eq.~\eqref{Eqn:2-coupling-1}). Otherwise, if $\vec{X}_t\not\in D$, then the choices of $X^{(1)}_{t+1}$ and $X^{(2)}_{t+1}$ are made independently (by Eq.~\eqref{Eqn:2-coupling-4}); but if $\vec{X}_t=(x_1,x_2)\in D$, then we first toss one coin that has probability $1-\frac{C_{x_1,x_2}}{2(n-1)}$ of landing Heads. 
	If the coin lands Tails, then we select $z$ uniformly at random from the set $N(x_1)\cap N(x_2)$, and \textit{both}  $X^{(1)}_{t+1}$ and $X^{(2)}_{t+1}$ are set equal to $z$; and if the coin lands Heads, then we select $X^{(1)}_{t+1}$ and $X^{(2)}_{t+1}$ independently according to the respective residual probabilities distributions $Q_{x_2}(x_1,\cdot)$ and $Q_{x_1}(x_2,\cdot)$ (Eqs.~\eqref{Eqn:2-coupling-2} and~\eqref{Eqn:2-coupling-3}).
	
	Once the two chains meet, they stay together forever.  Accordingly, 
	we define the \textit{coalescing time} $\tau$ to be the random variable 
	\[    \tau \;:=\;  \min\{t\geq 0:  X_t^{(1)}\,=\, X_t^{(2)} \}.
	\]
	Thus an important property of our chain $\{\vec{X}_t\}$ is that $X_t^{(1)}\neq X_t^{(2)}$ if and only if $\tau>t$.
	
	
	Now we shall define a slightly different Markovian coupling which will provide a sharper bound on the spectral gap in Section~\ref{sec3} than the bound obtained by a Type 1 coupling.  
	This version, when not in $D$, has the two components either both staying still or both moving.
	
	Given 
	the transition probability 
	matrix $\frak{P}=\frac{1}{2}(I+P)$, define a different Markov chain $\{(X_t^{(1)},X_t^{(2)})\}_{t\geq 0}$ on $V\times V$ with transition probability $\mathbf{P}$ as follows:
	\bigskip
	
	\begin{subequations}
		$\textbf{P}((w_1, w_2),(v_1, v_2))$\vspace{-1mm}\label{Eqn:2-m-coupling}
		\begin{numcases}{=}\label{Eqn:2-m-coupling-1}
		\frak{P}(w_1,v_1)\mathds{1}_{\{v_1 = v_2\}}, & \text{if $w_1 = w_2$;}\\\label{Eqn:2-m-coupling-2}
		\left(1-\frac{C_{w_1,w_2}}{2(n-1)}\right)Q_{w_2}(w_1,v_1)Q_{w_1}(w_2,v_2), & \text{if $(w_1,w_2)\in D\backslash D_0$ and $v_1\neq v_2$;}\\\label{Eqn:2-m-coupling-3}
		\left(1-\frac{C_{w_1,w_2}}{2(n-1)}\right)Q_{w_2}(w_1,v_1)Q_{w_1}(w_2,v_1)+\frac{C_{w_1,w_2}}{2(n-1)}U_{w_1,w_2}(v_1), & \text{if $(w_1,w_2)\in D\backslash D_0$ and $v_1= v_2$;}\\\label{Eqn:2-m-coupling-4}
		\frac{1}{2}P(w_1,v_1)\mathds{1}_{\{w_2=v_2\}}+\frac{1}{2}P(w_2,v_2)\mathds{1}_{\{w_1=v_1\}}, & \text{otherwise.}
		\end{numcases}
	\end{subequations}
	It can be verified that this Markov chain is a coupling
	of two chains on $V$ with matrix $\frak{P}$. We call it the \textit{Type 2 coupling for $G$ with respect to $D$}.

	The following result is crucial as it will be used throughout this paper.
	
	\begin{theorem}\label{Thm:lower bound for the gap}
		Let $G$ be a connected graph on $n$ vertices, and $D$ be a set of pairs of vertices that have at least one common neighbour, i.e.\ $N(v)\cap N(w)\neq \emptyset$ for every $(v,w)\in D$. Consider the Type 1 (or Type 2) coupling $\{(X_t^{(1)},X_t^{(2)})\}_{t\geq 0}$ for $G$ with respect to $D$. Suppose that there exist $k>0$ and $0<A<1$ such that 
		\[ \Pr\left(\left. (X_{t+k}^{(1)},X_{t+k}^{(2)})\in D \cup D_0\,\right|\,  (X_{t}^{(1)},X_{t}^{(2)})=(x_1,x_2)\right)\geq A \hspace{5mm}\hbox{for all }(x_1,x_2)\not\in D.
		\]
		Let $C^* = \min\{C_{v,w} : (v,w)\in D\}$. Then
		\[
		\Gap(G) \geq \;  \frac{C^*A}{(n-1)(k+1)}.
		\]  
	\end{theorem}
	
	The proof will use the following lemma.
	
	\begin{lemma}\label{lem:taubound}
		Let $\{\vec{X}_t\}=\{(X^{(1)}_t,X^{(2)}_t)\}$ be a Markovian coupling of two Markov chains, where each $X^{(i)}_t$ has transition matrix $P$. Assume there exist a positive integer $J$ and a positive real $A_0<1$
		such that $\Pr\left(\left. \vec{X}_{t+J}\in D_0\,\right|\,\vec{X}_t=(x_1,x_2)\right)\geq A_0$ for every $(x_1,x_2)\not\in D_0$.  
		Then Gap$(P)\geq A_0/J$.
	\end{lemma}
	\begin{proof}[Proof of Lemma \ref{lem:taubound}]
		Recall the coalescing time $\tau=\min\{t\geq 0: \vec{X}_t\in D_0\}$.
		Then our hypothesis says that 
		\[ \Pr\left(\tau>t+J\, \left|\,\tau>t,\vec{X}_t=(x_1,x_2)\right.\right) \;\leq \; 1-A_0
		\hspace{7mm}\hbox{for every $t\geq 0,\; (x_1,x_2)\not\in D_0$}.
		\]
		This implies that 
		\[   \Pr\left(\tau>Js\,\left|\, \vec{X}_0=(x_1,x_2)\right.\right)  \;\leq\;
		(1-A_0)^s     \hspace{5mm}\hbox{for every positive integer $s$.}
		\]
		
		For any choice of initial distribution $\vec{\mu}$, consider a coupling in which the initial distributions of $X_0^{(1)}$ and $X_0^{(2)}$ are $\vec{\mu}$ and $\vec{\pi}$ respectively.
		Let $s=\lfloor t/J\rfloor$. 
		Then Lemma \ref{lem:jointbound} implies that 
		\begin{eqnarray*}  
			\frac{1}{2}\|  \vec{\mu}\,{P}^t \,-\, \vec{\pi}\|_1  \;=\;  \frac{1}{2}\|  \vec{\mu}\,{P}^t \,-\, \vec{\pi}\,{P}^t\|_1   & \leq &   \Pr\left(X^{(1)}_t \neq X^{(2)}_t \right)
			\\
			& \leq & \Pr(\tau>t)
			\\
			& \leq & \Pr(\tau>sJ)
			\\
			& \leq & \left( \left(1-A_0 \right)^{\frac{1}{J}}\,\right)^{sJ}.  
		\end{eqnarray*}
		Now take $t^{th}$ roots, and let $t\rightarrow \infty$. By (\ref{eq.limspecradius}),  we obtain
		$\lambda_2({P})   \,  \leq \,   \left(1-A_0 \right)^{1/J}.$
		Using this inequality and Equation (\ref{eq.rootineq}), we find 
		\[   1-\lambda_2(P) \;\geq \; 1 \,-\,(1-A_0)^{1/J} \;\geq \; \frac{A_0}{J}.
		\]
		This proves the lemma.
	\end{proof}
	
	
	\begin{proof}[Proof of Theorem \ref{Thm:lower bound for the gap}]
		We consider the Type 1 (or Type 2) coupling $\{(X_t^{(1)},X_t^{(2)})\}_{t\geq 0}$ for $G$ with respect to $D$. Let $(x_1,x_2)\not\in D$. Then
		\begin{align*}
		&\Pr\left(\left.X_{t+k+1}^{(1)}=X_{t+k+1}^{(2)}\,\right|\,(X_t^{(1)},X_t^{(2)})=(x_1,x_2) \right) \\
		\geq &\; \sum_{(w_1,w_2)\in D\cup D_0} \mathbf{P}^{k}((x_1,x_2),(w_1,w_2)) \, \Pr\left(\left. X_{t+k+1}^{(1)}=X_{t+k+1}^{(2)}\,\right|\,(X_{t+k}^{(1)},X_{t+k}^{(2)})=(w_1,w_2) \right) \\
		&\hspace{15mm}\left(\hbox{recall $\mathbf{P}^{k}((x_1,x_2),(w_1,w_2))\,=\, \Pr \left(\left. (X_{t+k}^{(1)},X_{t+k}^{(2)})=(w_1,w_2)\,\right|\,(X_t^{(1)},X_t^{(2)})=(x_1,x_2) \right)$}\right)\\ 
		\geq &\; \sum_{(w_1,w_2)\in D\cup D_0} \textbf{P}^{k}((x_1,x_2),(w_1,w_2)) \,\frac{C^*}{2(n-1)}
		\hspace{15mm}\hbox{(by Equations (\ref{Eqn:2-coupling-1}), (\ref{Eqn:2-coupling-3}), (\ref{Eqn:2-m-coupling-1}), and (\ref{Eqn:2-m-coupling-3}))}
		\\
		= & \; \frac{C^*}{2(n-1)}\Pr\left(\left. (X_{t+k}^{(1)},X_{t+k}^{(2)})\in D\cup D_0\,\right|\,(X_t^{(1)},X_t^{(2)})=(x_1,x_2) \right)\\
		\geq & \; \frac{C^*A}{2(n-1)}.
		\end{align*} 
		
		Now apply Lemma \ref{lem:taubound} with $P$ replaced by $\frak{P}$, $J=k+1$, and $A_0=C^*A/2(n-1)$.  We conclude that 
		\[    \mbox{Gap}(\frak{P})\,\geq \,\frac{C^*A}{2(n-1)(k+1)}\,.
		\]
		Finally, Equation (\ref{eq.lazygap}) completes the proof by removing the factor of 2 from the denominator.
	\end{proof}

	\section{First results using Type 1 coupling}\label{sec3}
	As a first application of the coupling methodology, we prove the following.
	
	The \textit{join} of two disjoint graphs $G_1$ and $G_2$ is the graph obtained by adding all possible edges between $G_1$ and $G_2$. We note that a graph $G$ is disconnected if and only if $\overline{G}$ is a join of some graphs.
	
	\begin{theorem}
		\label{thm.joingap}
		Let $G=G_1\vee G_2$ be the join of $G_1$ and $G_2$, where $G_i$ is a graph with $n_i$ vertices ($i=1,2$). Let $n=n_1+n_2$. Then
		\[   \Gap(G)  \;\geq \;   \frac{\min\{n_1,n_2\}}{32(n-1)} \,.
		\]
	\end{theorem}
	\begin{proof}
		To prove the theorem, we shall assume without loss of generality that $n_1\leq n_2$.
		We shall write $V_i$ for $V(G_i)$ ($i=1,2$). Let $\{(X_t^{(1)},X_t^{(2)})\}_{t\geq 0}$ be the Type 1 coupling for $G$ with respect to $V_2\times V_2$. 
		
		We claim 
		\begin{equation}
		\label{eq.joinV2bound2}
		\Pr\left(\left. (X_{t+1}^{(1)},X_{t+1}^{(2)})\in V_2\times V_2\,\right|\,(X_t^{(1)},X_t^{(2)})=(w_1,w_2) \right) \;\geq \; \frac{1}{16}  
		\hspace{6mm}\hbox{for all }(w_1,w_2) \not \in V_2\times V_2.
		\end{equation} 
		To prove this claim, recall Equation~\eqref{Eqn:2-coupling}. 
		If $X_t^{(i)}\in V_2$, then the probability that $X^{(i)}_{t+1}$ stays in $V_2$ is at least $1/2$ (by the lazy chain property). And if $X_t^{(i)}\in V_1$, then the probability that $X^{(i)}_{t+1}\,\in\,V_2$ is at least $\frac{1}{2}\,\frac{n_2}{n} \,\geq\, \frac{1}{4}$
		(recall that $n_1\leq n_2$).
		Equation~(\ref{eq.joinV2bound2}) follows. Since any two vertices in $V_2$ have at least $n_1$ common neighbours, the conclusion follows from Theorem~\ref{Thm:lower bound for the gap}.
	\end{proof}
	
	\begin{remark}  \label{remark:join}
		It is found in \cite[Theorem 4.1]{kim2023bounds} that 
		$\mathcal{K}(J)\leq 3n$ for any join graph $J$, and hence Equation (\ref{eq.kemenygap}) implies that $\Gap(J)\geq \frac{1}{3n}$.  This bound is better than Theorem \ref{thm.joingap} only when $n_1$ is small.  Conversely, Theorem \ref{thm.joingap} never yields an improved bound on $\mathcal{K}(J)$ via Equation (\ref{eq.kemenygap}).
	\end{remark}
	
	\begin{theorem}
		\label{thm.largedegree}
		Let $G$ be a connected graph with $\Delta(G) = n-O(1)$. Then,  $\Gap(G) = \Omega\left(\frac{1}{n}\right)$.
	\end{theorem}
	\begin{proof}
		Suppose that $\Delta(G)\geq n-K$ for some constant $K>0$ (independent of $n$).	Let $v$ be a vertex with the maximum degree, and let 
		$\{(X_t^{(1)},X_t^{(2)})\}_{t\geq 0}$ be the Markov chain of the Type 1 coupling  with respect to $N(v)\times N(v)$. From Theorem~\ref{Thm:lower bound for the gap}, it suffices to show that 
		there exist $k>0$ and $A>0$ (which can depend on $K$ but not on $n$) such that
		$$\Pr\left(\left. (X_{t+k}^{(1)},X_{t+k}^{(2)})\in \left(N(v){\times} N(v)\right)\cup D_0 \,\right|\,  (X_{t}^{(1)},X_{t}^{(2)})=(x_1,x_2) \right)\;\geq\; A\quad\text{for all $(x_1,x_2)\not\in  N(v){\times} N(v)$.}$$

		Let $U = V(G)\backslash (N(v)\cup \{v\})$. 
		Then $U$ has at most $K$ vertices. Choose a vertex $u$ in $U$. Let $u_0-u_1-\cdots-u_p$ be a shortest path from $u=u_0$ to a vertex in $N(v)$. Here $u_p\in N(v)$. Note that $p \leq K$. For $j\leq p-2$, since the path is shortest, $u_j$ is in $U$ and  has no neighbours in $N(v)$. Since $\deg(u_j)<K$, 
		we have $\frak{P}(u_j,u_{j+1})>\frac{1}{2K}$ for $0\leq j\leq p-2$. Let $r$ be the number of neighbours of $u_{p-1}$ in $N(v)$. Then the probability that a step from $u_{p-1}$ goes to $N(v)$ is at least $r/2(r+K-1)$, which is greater than or equal to $1/2K$. Therefore 
		\[  \Pr\left(\left. X_{t+p}^{(i)}\in N(v) \,\right|\, X_{t}=u\right)  \;>\; \left(\frac{1}{2K}\right)^p.
		\]
		And since  $\frak{P}(u,u)=\frac{1}{2}\geq \frac{1}{2K}$, we can conclude that
		$$\Pr\left(\left. X_{t+K}^{(i)}\in N(v) \,\right|\,  X_{t}^{(i)}=u\right)\;>\; \left(\frac{1}{2K}\right)^K\quad\quad\text{for every $u\in U$ and $i=1,2$}.$$
		Also, for $x\in N(v)$, we have
		$$\Pr\left(X^{(i)}_{t+s}=x \hbox{ for }s=1,\ldots,K\,\left|\, X^{(i)}_t=x\right. \right) \;=\;  \frac{1}{2^K} \;\geq \; \frac{1}{(2K)^K}  $$
		and
		$$\Pr\left( X^{(i)}_{t+s}=v \hbox{ for }s=1,\ldots,K-1,\, X^{(i)}_{t+K}\in N(v)\,\left|\, X^{(i)}_t=v\right. \right) \;=\; \frac{1}{2^K} \;\geq \; \frac{1}{(2K)^K}.  $$
		Now we may use $A = \left(\frac{1}{2K}\right)^{2K}$ and $k=K$. This completes the proof.
	\end{proof}

	We observe that we are now able to prove the following limited Nordhaus-Gaddum result.  This result will be superseded by Theorem \ref{thm:main1}, but it is of interest to note that it is an easy consequence of what we have done so far. 
	
	\begin{corollary}
		\label{cor.largedegree}
		Let $G$ be a graph with $\Delta(G) = n-O(1)$ or $\delta(G)=O(1)$. Then  $\max\{\Gap(G),\Gap(\overline{G})\} = \Omega\left(\frac{1}{n}\right)$.
	\end{corollary}
	\begin{proof}
		If $G$ is disconnected, then $\overline{G}$ is a join graph, and the result follows from Theorem \ref{thm.joingap}.  Similarly the result holds if $\overline{G}$ is disconnected.  So we can assume that both $G$ and $\overline{G}$ are connected.  If $\Delta(G)=n-O(1)$, then the result follows directly from Theorem \ref{thm.largedegree}.  Finally, if $\delta(G)=O(1)$, then we apply Theorem \ref{thm.largedegree} to $\overline{G}$, using $\Delta(\overline{G})=n-1-\delta(G)$.  
	\end{proof}
	
	\section{Proof of Theorem~\ref{thm:main1}}\label{sec4}
	In this section, we prove Theorem~\ref{thm:main1}, leveraging the graph's diameter. Let $G$ be a graph on $n$ vertices. If $\diam(G)=1$, then $G$ is the complete graph and 
	$$\Gap(G) = 1+\frac{1}{n-1}>\frac{3-\sqrt{5}}{8}\frac{1}{n-1}.$$
	We note a basic fact in graph theory that the complement of a graph with diameter at least $4$ has diameter $2$, and the complement of a graph with diameter $3$ has diameter at most $3$. Hence it suffices to show that $\Gap(G)\geq \frac{3-\sqrt{5}}{8}\frac{1}{n-1}$ for two cases: (i)~$\diam(G)=2$ and (ii)~$\diam(G)=\diam(\overline{G})=3$. Propositions~\ref{diameter 2} and~\ref{diameter 3} below establish the former and latter case, respectively. 
	
	
	Recall that given a connected graph $G$, $D_i$ is the set of all pairs of vertices whose distance from each other is $i$, and $C_{v,w}=|N_G(v)\cap N_G(w)|$. 
	\begin{proposition}\label{diameter 2}
		Let $G$ be a connected graph of diameter $2$ on $n$ vertices. 
		Let $C^*=\min\{C_{v,w}:(v,w)\in D_2\}$.  Then
		$$\Gap(G)\;\geq \; \frac{min\{C^*,2\} }{n-1}.
		$$
	\end{proposition}
	\begin{proof}
		Consider the Type 2 coupling of Equation (\ref{Eqn:2-m-coupling}) $\{(X_t^{(1)},X_t^{(2)})\}_{t\geq 0}$ for $G$ with respect to $D_2$. Let $X_t^{(1)} = v$ and $X_t^{(2)} = w$ with $v\neq w$. 
		On the one hand, if $(v,w)\notin D_2\cup D_0$ (that is, $v$ is adjacent to $w$), then
		Equation (\ref{Eqn:2-m-coupling-4}) tells us that
		\begin{align*}
		\Pr(X_{t+1}^{(1)} = X_{t+1}^{(2)}\;|\;X_{t}^{(1)} = v,\;X_{t}^{(2)} = w)\;
		= &\; \frac{1}{2}\,P(v,w) \,+\, \frac{1}{2}\,P(w,v)
		\\
		=  &\; \frac{1}{2}\frac{1}{\deg(w)}+\frac{1}{2}\frac{1}{\deg(v)}
		\quad \geq \;\frac{1}{n-1}.
		\end{align*}  
		On the other hand, if $(v,w)\in D_2$, then 
		\[   \Pr(X_{t+1}^{(1)} = X_{t+1}^{(2)}\;|\;X_{t}^{(1)} = v,\;X_{t}^{(2)} = w)\;  \geq \;   \frac{C_{v,w}}{2(n-1)}   \;\geq \;
		\frac{C^*}{2(n-1)}\,.
		\]
		Using the above two inequalities in Lemma \ref{lem:taubound} with $J=1$,
		we see that Gap$(\frak{P})\geq \min\{C^*,2\}/(2(n-1))$.  The proposition 
		now follows from Equation (\ref{eq.lazygap}).
	\end{proof}
	
	\begin{remark}\label{rem:diam2}
		Consider the graph $G$ that consists of two cliques of equal size that have exactly one vertex in common. Then $\diam(G) = 2$ and $C^*=1$. Letting $n = |V(G)|$, by  \cite[Proposition~3.4]{faught2023nordhaus}, we obtain $$\Gap(G) = \frac{2}{n-1} \;>\; \frac{C^*}{n-1}.$$
		We see that there is a family of graphs whose spectral gap is greater than the lower bound in Proposition~\ref{diameter 2}.
	\end{remark}
	
	\begin{proposition}\label{diameter 3}
		Let $G$ be a connected graph with $\mathrm{diam}(G) = \mathrm{diam}(\overline{G}) =3$. Then
		$$\max\{\Gap(G), \Gap(\overline{G})\} \;\geq\;  \frac{3-\sqrt{5}}{8}\frac{1}{n-1}.
		$$   
	\end{proposition}
	\begin{proof}
		Let $A$ be a constant such that $0<A<1/4$, whose value we shall fix at the end of the proof. 
		Let $\{(X_t^{(1)},X_t^{(2)})\}_{t\geq 0}$ be the Markov chain of the Type 2 coupling for $G$ with respect to $D$, where $D$ is the set of all pairs of vertices with at least one common neighbour in $G$. Then $D_2\cup D_0\subseteq D$. 
		We consider two cases.
		
		\underline{Case 1}: Suppose 
		$$\Pr\left((X_{t+1}^{(1)},X_{t+1}^{(2)})\in D\;|\; X_t^{(1)} = v_1,X_t^{(2)} = v_2\right)\;>\;A\quad\text{for all $v_1$ and $v_2$ with $\mathrm{dist}_G(v_1,v_2) = 3$.}$$
		Consider $(v_1,v_2)\not \in D$ such that $\dist_G(v_1,v_2)\neq 3$.  Then $\dist_G(v_1,v_2)=1$ and $N_G(v_1)\cap N_G(v_2)=\emptyset$.
		If $X^{(1)}_t=v_1$ and $X^{(2)}_t=v_2$, then we see from Equation (\ref{Eqn:2-m-coupling-4}) that either $X^{(1)}_{t+1}=v_1$ and $X^{(2)}_{t+1}\in N_G(v_2)$, or else $X^{(1)}_{t+1}\in N_G(v_1)$ and $X^{(2)}_{t+1}=v_2$. In either case $(X^{(1)}_{t+1},X^{(2)}_{t+1})\in D$ because the components have either $v_2$ or $v_1$ as a common neighbour. 
		Then applying Theorem~\ref{Thm:lower bound for the gap} yields $$\Gap(G) \;\geq \;\frac{A}{2(n-1)}.$$ 
		
		\underline{Case 2}: Suppose Case 1 does not hold. Then there exist $v$ and $w$ with $\mathrm{dist}_G(v,w) =3$ such that $$\Pr\left((X_{t+1}^{(1)},X_{t+1}^{(2)})\in D\;|\; X_t^{(1)} = v,X_t^{(2)} = w\right)\;\leq \;A.$$ 
		Without loss of generality, assume $|N_G(v)|\leq |N_G(w)|$. Define $N_G^w(v)$ (resp. $N_G^v(w)$) to be the set of vertices $z$ in $N_G(v)$ (resp. $N_G(w)$) such that $\mathrm{dist}_G(z,w) = 2$ (resp. $\mathrm{dist}_G(z,v) = 2$). Since $D_2\subseteq D$, we now see from Equation (\ref{Eqn:2-m-coupling-4}) that
		$$\frac{|N_G^w(v)|}{2|N_G(v)|}+\frac{|N_G^v(w)|}{2|N_G(w)|}\;\leq\; \Pr\left((X_{t+1}^{(1)},X_{t+1}^{(2)})\in D_2\;|\; X_t^{(1)} = v,X_t^{(2)} = w\right)\;\leq \; A.$$
		
		Note that $\mathrm{dist}_{\overline{G}}(z_1,z_2) \geq 3$ if and only if $z_1$ and $z_2$ are adjacent in $G$ and each vertex $z$ is adjacent to $z_1$ or $z_2$ in $G$, equivalently, $N_G(z_1)\cup N_G(z_2) = V$. We claim that if $\mathrm{dist}_{\overline{G}}(x,y) = 3$ then $(x,y)$ or $(y,x)$ is in $N_G^w(v)\times N_G^v(w)$. Suppose that $\mathrm{dist}_{\overline{G}}(x,y) = 3$. Then  $\mathrm{dist}_{G}(x,y) = 1$ and $N_G(x)\cup N_G(y) = V$. If $v,w\in N_G(x)$ or $v,w\in N_G(y)$, then we would have $\mathrm{dist}_G(v,w) \leq 2$, which is a contradiction. Hence, one of $v$ and $w$ is in $N_G(x)\backslash N_G(y)$ and the other is in $N_G(y)\backslash N_G(x)$. Now, the claim is established.
		
		Consider $\overline{G}$, with diameter $3$. 
		Let $x$ and $y$ be two vertices such that $\mathrm{dist}_{\overline{G}}(x,y) = 3$.  From the claim above, one of $(x,y)$ or $(y,x)$ is in $N_G^w(v)\times N_G^v(w)$.  We shall first assume that $(x,y)$ is in this set, and return to the other possibility later.  Thus we are assuming that  
		$x\in N_G^w(v)$ and $y\in N_G^v(w)$.
		We notice that $N_{\overline{G}}(x) =N_G(y)\backslash (N_G(x)\cup\{x\})$ and $N_{\overline{G}}(y) =N_G(x)\backslash (N_G(y)\cup \{y\})$. 
		We claim that $N_{\overline{G}}(x)\backslash N_{\overline{G}}^y(x)$ is contained in $N_G^w(v)$. Pick $z\in N_{\overline{G}}(x)\backslash N_{\overline{G}}^y(x)$. For sake of contradiction, assume that $v$ is not adjacent to $z$ in $G$. Then $z$ is adjacent to $v$ in $\overline{G}$ so that there is a path $x-z-v-y$ in $\overline{G}$, that is, $\mathrm{dist}_{\overline{G}}(z,y)=2$. This is a contradiction to $z\notin N_{\overline{G}}^y(x)$. Hence, $z$ is adjacent to $v$ in $G$. Since $z\in N_G(y)$, there is a path $v-z-y-w$ in $G$. Therefore $z\in N_G^w(v)$, and we have proved that
		$$N_{\overline{G}}(x)\backslash N_{\overline{G}}^y(x)\subseteq N_G^w(v).$$ 
		
		Consider $z\in N_G(w)$. If $z\in N_G(x)$, then we have a path $w-z-x-v$ in $G$ and so $z\in N_G^v(w)$. Hence, $N_G(w)\cap N_G(x)\subseteq N_G^v(w)$. Since $N_G(x)\cup N_G(y) = V$ and $x\not\in N_G(w)$, we obtain $N_G(w)\backslash N_G^v(w)\subseteq N_G(y)\backslash N_G(x)$. Hence, $$N_G(w)\backslash N_G^v(w)\subseteq N_{\overline{G}}(x).$$

		Since $\frac{|N_G^w(v)|}{2|N_G(v)|}\leq A$, we have $|N_G^w(v)|\leq 2|N_G(v)|A$. Similarly, $|N_G^v(w)|\leq 2|N_G(w)|A$. We now see that \begin{align*}
		\frac{|N_{\overline{G}}(x)\backslash N_{\overline{G}}^y(x)|}{|N_{\overline{G}}(x)|}\;\leq\; 
		\frac{|N_G^w(v)|}{|N_G(w)|-|N_G^v(w)|} \;\leq\; \frac{|N_G^w(v)|}{(1-2A)|N_G(w)|}\;\leq \; \frac{|N_G^w(v)|}{(1-2A)|N_G(v)|} \;\leq \; \frac{2A}{1-2A}.
		\end{align*}
		Hence,
		$$\frac{|N_{\overline{G}}^y(x)|}{|N_{\overline{G}}(x)|}\;\geq \; 1-\frac{2A}{1-2A} \;>\; 0.$$
		
		Let $\{(\overline{X}_t^{(1)},\overline{X}_t^{(2)})\}_{t\geq 0}$ correspond to the Type 2 coupling for $\overline{G}$ with respect to $\overline{D}$, where $\overline{D}$ is the set of all pairs of vertices with at least one common neighbour in $\overline{G}$. Let $\overline{D}_2$ be the set of all pairs of vertices whose distance is $2$ in $\overline{G}$. Then $\overline{D}_2\subseteq \overline{D}$, and 
		\begin{align}
		\nonumber
		\Pr& \left(\left.\left(\overline{X}_{t+1}^{(1)},\overline{X}_{t+1}^{(2)}\right)\in  \overline{D}\;\right|\; \overline{X}_t^{(1)} = x,\overline{X}_t^{(2)} = y\right)\\
		\label{eq.PrdistA}
		& \geq\; \Pr\left(\left.\left(\overline{X}_{t+1}^{(1)},\overline{X}_{t+1}^{(2)}\right)\in \overline{D}_2\;\right|\; \overline{X}_t^{(1)} = x,\overline{X}_t^{(2)} = y\right)\;\geq \;\frac{1}{2}\left(1-\frac{2A}{1-2A}\right).
		\end{align}
		We have shown that Equation (\ref{eq.PrdistA}) holds for every $x$ and $y$ with $\mathrm{dist}_{\overline{G}}(x,y)=3$ such that $(x,y)\in N_G^w(v)\times N_G^v(w)$.
		Since the probabilities in Equation (\ref{eq.PrdistA}) do not change if $x$ and $y$ are interchanged, the inequality also holds if $\mathrm{dist}_{\overline{G}}(x,y)=3$ and $(y,x)\in N_G^w(v)\times N_G^v(w)$. As noted above, this proves that Equation (\ref{eq.PrdistA}) holds for every $(x,y)$ such that $\mathrm{dist}_{\overline{G}}(x,y)=3$.
		Now, using the argument of Case 1, we deduce that 
		$$\Gap(\overline{G}) \;\geq \;\left(\frac{1}{4}-\frac{A}{2(1-2A)}\right)\frac{1}{n-1}$$
		in Case 2.
		
		Finally, combining the two cases and choosing $A$ to satisfy $\frac{1}{4}-\frac{A}{2(1-2A)} = \frac{A}{2}$, i.e.\ $A=(3-\sqrt{5})/4$, yields the conclusion of the proposition.
	\end{proof}

	\section{Proof of Theorem~\ref{thm.complementLU}}\label{sec5}
	
	In this section, we examine Nordhaus--Gaddum bounds when vertex degrees are $\Omega(n)$ in both $G$ and $\overline{G}$. We begin with introducing the relation between the spectral gap and the so-called \textit{bottleneck ratio}.


	
	
	%
	%
	
	When $S$ and $T$ are disjoint subsets of $V$, we define $[S,T]_G$ to be the set of all edges of $G$ that have one endpoint in $S$ and one endpoint in $T$. For $S\subseteq V$, let 
	\[    {\rm vol}(S)\;:=\;  \sum_{v\in S} \deg_G(v) \,.
	\]
	Note that ${\rm vol}(V)\,=\,2|E(G)|$.  For example, in a regular
	graph of degree $d$, ${\rm vol}(S)=d|S|$.
	The \textit{bottleneck ratio} of the graph $G$ is  defined to be
	\begin{equation}
	\label{eq.defPhi}
	\Phi  \;=\;  \Phi(G)  \;=\;  \min_{S\subseteq V:  \,0<{\rm vol}(S)\leq |E(G)|} \frac{  |\,[S,S^c]_G|}{{\rm vol}(S)}  \,.
	\end{equation}
	The classic work of Jerrum and Sinclair \cite{JeSi} and Lawler and Sokal \cite{LaSo} proved 
	\begin{equation}
	\label{eq.jerrum}
	\frac{\Phi^2}{2}  \;\leq \;  1-\lambda_2  \;\leq \;  2\Phi.
	\end{equation}
	(Note that if $G$ is disconnected, then $\lambda_2=1$ and $\Phi=0$.)
	For an overview from the point of view of discrete reversible Markov chains, of which our
	context is a special case,
	see  \cite[Sections~7.2 and~13.3]{LPW}.
	


	Here we restate Theorem~\ref{thm.complementLU}.
	
	\noindent\textbf{Theorem~\ref{thm.complementLU}} \textit{Let $L$ and $U$ be real constants such that $0<L< U<1$, and let $\Gamma_{L,U}=L^4(1-U)^4/2^{13}$. Then for every integer $n\geq \frac{2}{1-U}$ and every
		graph $G$ on $n$ vertices satisfying $Ln\leq \delta(G)\leq \Delta(G)\leq Un$,
		\[    \max\left\{  \Gap(G), \, \Gap(\overline{G})\right\}   \;\geq \; \Gamma_{L,U}.
		\]}
	
	
	\bigskip
	
	\noindent\textbf{Proof of Theorem \ref{thm.complementLU}:}
	We first consider the case that $G$ is disconnected.  Let $C$ be a connected component of $G$ with fewest vertices.  Then $\frac{n}{2}\geq |V(C)|>\delta(G)\geq Ln$.  Therefore $\overline{G}$ can be expressed as a join of two graphs that each have at least $Ln$ vertices.  It follows from Theorem \ref{thm.joingap} that $\Gap(\overline{G})\geq \frac{L}{32}$.
	Similarly, if $\overline{G}$ is disconnected, then the number of vertices in each component is more than $\delta(\overline{G})=n-1-\Delta(G)$, and applying Theorem \ref{thm.joingap} to $G$ yields $\Gap(G)\geq \frac{1-U}{32}$. Thus the theorem holds for the class of disconnected graphs and their complements.  Henceforth, then, we shall assume that both $G$ and $\overline{G}$ are connected.
	
	We shall require a large parameter $M$ and a small parameter $\epsilon$ which we define by
	\begin{equation}
	\label{eq.defMeps}
	M \; :=\;  \frac{8}{L(1-U)} 
	\hspace{5mm} \text{and} \hspace{5mm}
	\epsilon \; := \;  \frac{1}{M^2}  \;=\;   \frac{L^2(1-U)^2}{64} \,.
	\end{equation}
	We also let 
	\begin{equation}
	\label{eq.Gammadef}
	\Gamma \;\equiv\;   \Gamma_{L,U} \; :=\;  \frac{\epsilon^2}{2}  \;=\;   \frac{L^4(1-U)^4}{2^{13}}\,.
	\end{equation}
	We shall show that, for the graphs under consideration,
	\begin{verse}
		If $\Gap(G)   \,<\, \Gamma$, then $\Gap(\overline{G}) \,\geq\, \Gamma$.
	\end{verse}
	This will prove the theorem.
	
	Assume $n\geq \frac{2}{1-U}$, and
	let $G$ be a graph with $n$ vertices such that 
	$\Gap(G)<\Gamma$ and $Ln\leq \delta(G)\leq \Delta(G)\leq Un$.  Then
	\begin{equation}
	\label{eq.degbound1}   
	Ln   \;\leq \;  \deg_G(v)  \; \leq \;      Un  \hspace{10mm} \forall v\in V \,.
	\end{equation}
	For future reference, we also note that
	\begin{equation}
	\label{eq.volSvol}
	\frac{ {\rm vol}(T)}{Un}  \;\leq \;  |T|  \;\leq \;  \frac{{\rm vol}(T)}{Ln}   \hspace{10mm} \forall \,T\subseteq V.
	\end{equation} 
	In the complement $\overline{G}$, we have $n-1-Un \leq \delta(\overline{G})\leq \Delta(\overline{G})\leq n-1-Ln  < (1-L)n$.
	Since $1\leq \,\frac{(1-U)n}{2}$, we have that
	\begin{equation}
	\label{eq.degbound1c}   
	\frac{(1-U)\,n}{2}   \;\leq \;  \deg_{\overline{G}}(v)  \; \leq \;      (1-L)\,n  \hspace{10mm} \forall \,v\in V \,.
	\end{equation}

	Here is an outline of the proof.  The first three steps are similar to those in the proof of an analogous result for Kemeny's constant in \cite{kim2023bounds}.
	
	\smallskip
	\noindent
	\textit{\underline{Step 1}:}
	Since $\Gap(G)$ is small, the bottleneck ratio of $G$ must be small.  This means
	that there is a set $S$ of vertices such that there are relatively few edges in the cut $[S,S^c]_G$.
	Moreover, we show that the sets $S$ and $S^c$ each have at least order $n$ vertices.
	The idea then is that the corresponding complementary cut 
	$[S,S^c]_{\overline{G}}$ contains most of the $|S|\,|S^c|$ possible edges between 
	$S$ and $S^c$. 
	
	\smallskip
	\noindent
	\textit{\underline{Step 2}:}  We specify two sets of vertices $A\subseteq S$ and $B\subseteq S^c$ such that every vertex in $A$
	is connected (in $\overline{G}$) to most of $S^c$, every vertex in $B$ is connected to most of $S$, and
	the set differences $S-A$ and $S^c-B$ are both relatively small.
	
	\smallskip
	\noindent
	\textit{\underline{Step 3}:} 
	We now look at properties of the random walk on $\overline{G}$.  
	We prove that there is a $\Theta>0$ (depending on $L$ and $U$ but not $n$) such that 
	from every vertex of $A$ (respectively, $B$), the probability of entering $B$ (respectively, $A$) on the 
	next step is at least $\Theta$.  We also prove that from any vertex, the probability that the next 
	step is to $A\cup B$ is not small (in fact, at least $1/2$).  
	
	\smallskip
	\noindent  
	\textit{\underline{Step 4}:}
	With little loss of generality, assume $|S|\geq |S^c|$.
	We use the Type 1 coupling for $\overline{G}$ with respect to $B\times B$ together with Theorem~\ref{Thm:lower bound for the gap}. Step 3 gives a constant lower bound on the probability that the two walks are reasonably likely to both be in $A\cup B$ after one step, and both in $B$ after a second step.
	This enables us to apply Theorem \ref{Thm:lower bound for the gap}.
	
	
	\medskip
	Now we return to the proof.

	\medskip
	\textit{\underline{Step 1}.}
	Since $\Gap(G)  \,<\, \Gamma$, it follows from Equations (\ref{eq.jerrum}) 
	and (\ref{eq.Gammadef}) that
	$\Phi(G) \; <  \;   \epsilon$.   
	Therefore, by the definition of $\Phi(G)$ in Equation (\ref{eq.defPhi}), there is a nonempty subset $S$ of $V$
	such that ${\rm vol}(S)\leq |E(G)|$ and 
	\begin{equation}
	\label{eq.SPhibeta}
	|\,[S,S^c]_G|   \;< \;  \epsilon \,{\rm{vol}}(S) \;\leq \;  \epsilon \,n\,U\,|S| 
	\end{equation}
	(the second inequality uses Equation (\ref{eq.volSvol})).  Also, we have
	\begin{equation*}
	|S^c|\,Un   \;\geq \;  {\rm vol}(S^c)   \;=\; 
	2|E(G)|-{\rm vol}(S)  \;\geq \;|E(G)| \;\geq \; \frac{L\,n^2}{2},
	\end{equation*}
	and hence
	\begin{equation}
	\label{eq.2UScn}
	\frac{2U}{L}\, |S^c| \;\geq \; n \,.
	\end{equation}
	
	Let $C\,=\,  (L-\epsilon U)/2$.   
	Since $\epsilon<L$ by  Equation (\ref{eq.defMeps}), we see that
	\begin{equation}
	\label{eq.Cbound}
	\frac{L}{2} \;>\;   C  \; >  \;  \frac{L(1-U)}{2}    \;> \; 0.
	\end{equation}
	We claim that $|S|\,>\,n C$.   If not, then every vertex $v$ in $S$ has at most 
	$nC$ neighbours in $S$ (in $G$), and hence 
	\[
	|\,[v,S^c]_G| \;  \geq    \; Ln \,-\, Cn \;> \;   \epsilon U n \,.  
	\]
	Summing the above inequality over all $v\in S$ gives $|\,[S,S^c]_G|  \,> \, \epsilon U n \,|S|$, which 
	contradicts Equation (\ref{eq.SPhibeta}).  This proves the claim that
	\begin{equation}
	\label{eq.SgtnC}
	|S| \;>\;  nC  \,  \hspace{5mm}\text{where }C\;=\;  \frac{L-\epsilon U}{2}.
	\end{equation}
	

	\medskip
	
	\textit{\underline{Step 2}.}
	Note that $M\epsilon \,=\, L(1-U)/8\,<\,1/8$ and hence
	\begin{equation}
	\label{eq.Mepsbd}
	1-M\epsilon   \;>  \;  \frac{7}{8}  \,.
	\end{equation}
	Define the sets of vertices $A\subseteq S$ and $B\subseteq S^c$ by
	\begin{eqnarray}
	\label{eq.Adef1}
	A  & = &  \{v\in S\,:\, |\,[v,S^c]_{\overline{G}}| \,\geq \, |S^c|(1-M\epsilon)    \,\}
	\\
	\label{eq.Adef2}
	& = &  \{v\in S\,:\, |\,[v,S^c]_{G}| \,\leq\, M\epsilon \, |S^c|   \, \} \, , \hspace{7mm} \text{and}
	\\
	\label{eq.Bdef1}
	B  & = &  \{w\in S^c\,:\, |\,[w,S]_{\overline{G}}| \,\geq \, |S|(1-M\epsilon)    \,\}
	\\
	\label{eq.Bdef2}
	& = &  \{w\in S^c\,:\, |\,[w,S]_{G}| \,\leq\, M\epsilon \, |S|   \, \}   \,.
	\end{eqnarray}
	From Equation (\ref{eq.Adef2}), it follows that 
	\[       |\,[S,S^c]_G| \;\geq \;  \sum_{v\in S-A}|\,[v,S^c]_G|    \;\;\geq \;\;  |S-A|\times M\epsilon |S^c|
	\]
	and hence (using Equations (\ref{eq.SPhibeta}) and (\ref{eq.volSvol}), 
	and the bounds ${\rm vol}(S)\leq |E(G)|\leq {\rm vol}(S^c)$)  that 
	\begin{equation}
	\label{eq.SminusAbound}
	|S-A|\;\leq \;   \frac{  |\,[S,S^c]_G|}{M\epsilon \,|S^c|}  \;\leq \; \frac{\epsilon \, {\rm vol}(S)}{M\epsilon \,({\rm vol}(S^c)/Un)}  
	\;\leq \;  \frac{nU}{M}\,.
	\end{equation} 
	Similarly, we have     
	\[       |\,[S,S^c]_G| \;\geq \;  \sum_{w\in S^c-B}|\,[w,S]_G|    \;\;\geq \; \;  |S^c-B|\times M\epsilon |S|
	\]
	and hence (using Equation (\ref{eq.SPhibeta}))   that 
	\begin{equation}
	\label{eq.ScminusBbound}
	|S^c-B|\;\leq \;   \frac{  |\,[S,S^c]_G|}{M\epsilon \,|S|}  \;\leq \;
	\frac{nU}{M}   \,.
	\end{equation} 
	By Equations  (\ref{eq.SminusAbound}), (\ref{eq.SgtnC}),  (\ref{eq.ScminusBbound}), and (\ref{eq.2UScn}),  we have
	\begin{eqnarray}
	\nonumber 
	|A|  &\geq &  |S|\,-\, n\frac{U}{M}  \;\; \geq \; n\left( C -  \frac{U}{M}  \right)   \quad\text{ and}
	\\
	\nonumber 
	|B|  &\geq &  |S^c|\,-\, n\frac{U}{M}  \;\; \geq \; n\left(  \frac{L}{2U} -  \frac{U}{M}  \right) \,.  
	\end{eqnarray}
	We note the above two lower bounds are strictly positive because
	\begin{equation}
	\label{eq.U2MLCbound}
	\frac{U}{M}\;=\;  \frac{UL(1-U)}{8}   \;<\;  \frac{L(1-U)}{8}  \; <\;  \frac{C}{4} \;<\; \frac{L}{8U} \hspace{6mm}\text{by Eq.\ (\ref{eq.Cbound})}.
	\end{equation}

	

	
	By Equation (\ref{eq.Mepsbd}), the definition of $A$ says that 
	$|[v,S^c]_{\overline{G}}| \,\geq \,  \frac{7}{8}\,|S^c|$ for every $v\in A$, and hence 
	\begin{equation}
	\label{eq.vvinA}
	\left|  N_{\overline{G}}(v_1)\cap N_{\overline{G}}(v_2)\cap S^c\right|   \;\geq \;  \frac{3}{4}\, |S^c|   \hspace{5mm}\mbox{for all }v_1,v_2\in A.
	\end{equation}
	Similarly, we have
	\begin{equation}
	\label{eq.wwinB}
	\left|  N_{\overline{G}}(w_1)\cap N_{\overline{G}}(w_2)\cap S\right|   \;\geq \;  \frac{3}{4}\, |S|   \hspace{5mm}\mbox{for all }w_1,w_2\in B.
	\end{equation}

	\medskip
	
	\textit{\underline{Step 3}.}
	In Step 3, we consider the random walk on the complement of $G$, i.e.\ having the transition probabilities $P_{\overline{G}}(v,w)=1/\deg_{\overline{G}}(v)$ if $\{v,w\}\in E(\overline{G})$.
	
	For $w\in B$, we have
	\begin{eqnarray}
	\nonumber
	P_{\overline{G}}(w,A)    & = &   |\,[w,A]_{\overline{G}}| \, / \, \deg_{\overline{G}}(w)
	\\
	\nonumber
	& \geq & \frac{ |\,[w,S]_{\overline{G}}| \,-\, |S-A| }{n(1-L)}  
	\hspace{15mm}\text{(by Eq.\  (\ref{eq.degbound1c}))}      
	\\
	\nonumber
	& \geq & \frac{ |S|(1-M\epsilon) \,-\,n\frac{U}{M}  }{n(1-L)}   
	\hspace{12mm}\text{(by  Eqs.\ (\ref{eq.Bdef1}) and  (\ref{eq.SminusAbound}))}
	\\
	\nonumber
	& \geq & \frac{ nC(1-M\epsilon) \,-\,n\frac{U}{M}  }{n(1-L)}   
	\hspace{12mm}\text{(by  Eq.\  (\ref{eq.SgtnC}))}    
	\\
	\label{eq.PwAbound}
	& = & \frac{ C(1-M\epsilon) \,-\,\frac{U}{M}  }{1-L} \quad  =:\;  \Theta. 
	\end{eqnarray}
	We know from  Equations (\ref{eq.Mepsbd}), (\ref{eq.U2MLCbound}), and (\ref{eq.Cbound}) that
	\begin{equation}
	\label{eq.ThetaCLU0}
	\Theta \;>\; \frac{7}{8}C-\frac{C}{4}  \;>\; \frac{C}{2} \;>\; \frac{L(1-U)}{4} \;>\; 0.
	\end{equation}
	Similarly, for $v\in A$,
	\begin{eqnarray}
	\nonumber
	P_{\overline{G}}(v,B)    & = &   |\,[v,B]_{\overline{G}}| \, / \, \deg_{\overline{G}}(v)
	\\
	\nonumber
	& \geq & \frac{ |\,[v,S^c]_{\overline{G}}| \,-\, |S^c-B| }{n(1-L)}   
	\hspace{15mm}\text{(by Eq.\  (\ref{eq.degbound1c}))}
	\\
	\nonumber
	& \geq & \frac{ |S^c|(1-M\epsilon) \,-\,n\frac{U}{M}  }{n(1-L)}   
	\hspace{15mm}\text{(by Eqs.\ (\ref{eq.Adef1}) and  (\ref{eq.ScminusBbound}))}
	\\
	\nonumber
	& \geq & \frac{ \frac{nL}{2U}(1-M\epsilon) \,-\,n\frac{U}{M}  }{n(1-L)}   
	\hspace{15mm}\text{(by Eq.\ (\ref{eq.2UScn}))}         
	\\
	\nonumber
	& \geq  & \frac{ C(1-M\epsilon) \,-\,\frac{U}{M}  }{1-L}   \hspace{15mm}\text{(by Eq.\ (\ref{eq.U2MLCbound}))}
	\\
	\label{eq.PvBbound}
	& = &  \; \Theta.
	\end{eqnarray}

	For every $z\in V$,   
	\begin{eqnarray}
	\nonumber
	P_{\overline{G}}(z,(A\cup B)^c)   & \leq &   (|S-A|\,+\,|S^c-B|) \,/\, \deg_{\overline{G}}(z)
	\\
	\nonumber
	& \leq & \frac{ 2nU/M }{n(1-U)/2} 
	\hspace{10mm}\text{(by Eqs.\ (\ref{eq.SminusAbound}), (\ref{eq.ScminusBbound}),
		and (\ref{eq.degbound1c}))}
	\\
	\nonumber   
	& = &   \frac{4U}{M(1-U)} \,.
	\end{eqnarray}
	Therefore
	\begin{equation}
	\label{eq.PzABbound}
	P_{\overline{G}}(z, A\cup B)   \;\geq \;     1\,-\,    \frac{4U}{M(1-U)}   \;>\;  \frac{1}{2}  \hspace{9mm}
	\text{for every vertex $z$}
	\end{equation}
	(the final inequality holds because $\frac{4U}{M(1-U)} \,=\, \frac{UL}{2}\,<\,\frac{1}{2}$).

	\medskip
	
	\textit{\underline{Step 4}.} First, suppose that $|S|\geq n/2$.  We now consider the Type 1 coupling $\{(\overline{X}_t^{(1)},\overline{X}_t^{(2)})\}_{t\geq 0}$ for $\overline{G}$ with respect to $B\times B$. 
	If the chain enters $D_0$, then it stays in $D_0$.  Otherwise, if $\overline{X}_t^{(i)}\in B$, then the probability that $\overline{X}^{(i)}_{t+1}$ stays in $B$ is at least $1/2$ (by the lazy chain property); if $\overline{X}_t^{(i)}\in A$, then the probability that $\overline{X}^{(i)}_{t+1}\,\in\,B$ is at least $\frac{\Theta}{2}$ by Eq.~\eqref{eq.PvBbound}; and if $\overline{X}_t^{(i)}\notin A\cup B$, then the probability that $\overline{X}^{(i)}_{t+2}\,\in\,B$ is at least $\left(\frac{1}{2}\Theta\right)\left(\frac{1}{2}\times\frac{1}{2}\right)\,=\,\frac{\Theta}{8}$ by Eqs.~\eqref{eq.PvBbound} and~\eqref{eq.PzABbound}. Hence
	\begin{align}
	\label{eq.PBBTheta}
	\Pr\left(\left. \left(\overline{X}_{t+2}^{(1)},\overline{X}_{t+2}^{(2)}\right)\in (B{\times} B)\cup D_0 \,\right|\,   \left(\overline{X}_{t}^{(1)},\overline{X}_{t}^{(2)}\right) = (x_1,x_2)\right)\;\geq\; \frac{\Theta^2}{64}  \hspace{5mm}\hbox{for all $(x_1,x_2)\not\in B\times B$}.
	\end{align}
	We shall now apply Theorem~\ref{Thm:lower bound for the gap}.  By Equation \eqref{eq.wwinB}, any two vertices in $B$ have at least $\frac{3}{4}|S|$ common neighbours in $\overline{G}$, so that $C^*\geq \frac{3}{4}|S|$.
	Combining this with  Equations (\ref{eq.PBBTheta}) and (\ref{eq.ThetaCLU0}), we obtain
	$$\Gap(\overline{G}) \;\geq \;\frac{3|S|}{4}\frac{\Theta^2}{64}\frac{1}{3(n-1)}\;\geq\; \frac{\Theta^2}{2^9}\;\geq\; \frac{L^2(1-U)^2}{2^{13}}\;\geq\; \Gamma.$$

	Alternatively, suppose that $|S^c|\geq n/2$. Then the argument and the bounds are exactly the same as in the case $|S|\geq n/2$, with the roles of $S$ and $S^c$ interchanged, and $A$ and $B$ interchanged.  
	
	Thus the  theorem is proved. \hfill  $\Box$
	
	\section{Example: $K_{m,m}$ with an attached lollipop}\label{sec6}
	
	In this section, we shall provide a family of graphs that disproves Conjecture~\ref{conjecture2}.
	
	The \textit{complete bipartite graph}, denoted $K_{s_1,s_2}$, is the join of two empty graphs of order $s_1$ and $s_2$. 
	The \textit{lollipop graph}, denoted $L_{r,r}$, is the graph formed by taking the clique $K_r$ of size $r$ and the path $P_r$ on $r$ vertices, and joining a vertex of the clique and an end-vertex of the path by an edge. Let $G_{r,m}$ 
	be the graph obtained from $L_{r,r}$ and $K_{m,m}$ by joining the vertex of degree $1$ in $L_{r,r}$ and a vertex of $K_{m,m}$ by an edge; see Figure~\ref{fig:graph G}.

	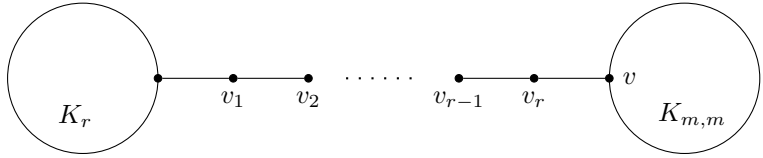
\begin{figure}[h!]
		\centering
		\begin{tikzpicture}
		\tikzset{enclosed/.style={draw, circle, inner sep=0pt, minimum size=.10cm, fill=black}}
		
		\node[enclosed,label={below, : }] (v0) at (-3,0) {};
		\node[enclosed, label={below,  : $v_1$}] (v1) at (-2,0) {};
		\node[enclosed, label={below, : $v_2$}] (v2) at (-1,0) {};
		\node[enclosed, label={below,  : $v_{r-1}$}] (vr-1) at (1,0) {};
		\node[enclosed, label={below,  : $v_r$}] (vr) at (2,0) {};
		\node[enclosed, label={right,  : $v$}] (v) at (3,0) {};

		\draw (-4,0) circle (1cm);
		\draw[thick, loosely dotted] (-.5,0)--(.5,0);
		\draw (4,0) circle (1cm);

		\draw (v0) -- (v1); \draw (v1)--(v2);
		\draw (vr-1) -- (vr); \draw (vr)--(v);
		
		\node[] at (-4.1,-0.5) {$K_r$};
		\node[] at (4.1,-0.5) {$K_{m,m}$};
		\end{tikzpicture}
		\caption{An illustration of the graph $G_{r,m}$ in Lemma~\ref{lem:graph G}.}
		\label{fig:graph G}
	\end{figure}
	
	We introduce a useful tool to find an upper bound on the spectral gap of $G$. Recall that the stationary distribution of
	the random walk on $G$ is given by $\pi(v)=\frac{\deg(v)}{2|E(G)|}$. It appears in \cite[Section~3.6]{aldous2002reversible} that
	\begin{align}\label{eqn: gap}
	\Gap(G) =&\; \inf_{F\in S}\frac{\sum_i\sum_{j\neq i}\pi(i)P(i,j)(F(i)-F(j))^2}{2\sum_{i}\pi(i)(F(i))^2} = \inf_{F\in S}\frac{\sum_{\{i,j\}\in E(G)}(F(i)-F(j))^2}{2\sum_{i}\deg(i)(F(i))^2}
	\end{align}
	where $S = \{F:V\rightarrow \mathbb{R}\mid \sum_{i\in V(G)}\pi(i)F(i) = 0\}$.
	
	\begin{lemma}\label{lem:graph G}
		Let $G_{r,m}$ be the graph as defined above. Suppose that $m\geq 2r$. Then
		$$\Gap(G_{r,m}) <\frac{4}{r^3}.$$
	\end{lemma}
	\begin{proof}	
		We regard $K_r,P_r,K_{m,m}$ as the subgraphs of $G_{r,m}$. For $k=1,\ldots,r$, let $v_k$ be the vertex of $P_r$ whose distance from $K_r$ is $k$. 
		Set $a = 1$ and $b = (r^2+1)/(2m^2+r+1)$. Define the function $F:V\rightarrow \mathbb{R}$ by 
		\[    F(x)    \;=\;  
		\begin{cases}
		-a,  & \text{if $x\in V(K_r)$;} \\
		-a + \frac{b+a}{r+1}k, & \text{if $x = v_k$ for $k=1,\dots,r$;}\\
		b, & \text{if $x\in V(K_{m,m})$.}
		\end{cases}
		\]
		It can be seen that $\sum_{i\in V(G)}\pi(i)F(i) = 0$. Note that $0<b<1$ and $a+b<2$. Then
		\begin{align*}
		\sum_{\{i,j\}\in E(G)}(F(i)-F(j))^2 \;=\; \frac{(b+a)^2}{r+1} 
		\;<\;\frac{4}{r}.
		\end{align*}
		Note that since $4r+2\geq 3r+3$, we have $(2r+1)/3(r+1)\geq 1/2$.
		Moreover,
		\begin{align*}
		&\sum_{i}\deg(i)(F(i))^2 \\
		=&\; a^2(r-1)^2+a^2r+\sum_{k=1}^r 2\left(-a+\frac{b+a}{r+1}k\right)^2 + b^2(m+1)+b^2m(2m-1)\\
		= &\; r^2-r+1+b^2\left(2m^2+1+\frac{r(2r+1)}{3(r+1)}\right)+b\left(-2r+\frac{2r(2r+1)}{3(r+1)}\right)+\frac{r(2r+1)}{3(r+1)}\\
		>&\;r^2-r+1+b\left(-2r+r\right)+\frac{r}{2} \\>&\;r^2-\frac{3}{2}r+1.
		\end{align*}
		Since $r^3-3r^2+2r\geq 0$ for every nonnegative integer $r$, by \eqref{eqn: gap}, we have \begin{equation*}
		\Gap(G)< \frac{4}{2r^3-3r^2+2r} =  \frac{4}{r^3+r^3-3r^2+2r}\leq \frac{4}{r^3}. \qedhere
		\end{equation*}
	\end{proof}
	\begin{remark}
		The technique in the proof above is a modified version of the one in the proof for the spectral gap of the barbell graph in \cite[Section~5]{aldous2002reversible}, which is the graph obtained from $G$ by replacing $K_{m,m}$ by $K_m$ where $m=r$. If we change the condition $m=r$ to $r/m\rightarrow 0$ as $n\rightarrow \infty$, then one can check that applying a similar argument, the spectral gap for the ``unbalanced'' barbell graph is bounded above by $1/r^3$ up to constant. 
	\end{remark}
	
	Consider $G=G_{r,m}$, with $n=2r+2m=|V(G)|$. 
	Let $v$ be the vertex of $K_{m,m}$ in $G$ that is adjacent to $P_r$. 
	The complement $\overline{G}$ contains two cliques of size $m$, which we shall call $A$ and $B$, where $A$ is the one containing $v$. Each of the vertices in the complement of $L_{r,r}$ (viewed as a subgraph of $\overline{G}$) is connected in $\overline{G}$ to every vertex in $B$. Then we see from \eqref{eq.jerrum} that
	\begin{equation*}
	\label{eq.Kmmbarupper}
	\Gap(\overline{G})  \;\leq \;   \frac{2[B,B^c]_{\overline{G}}}{\hbox{vol}_{\overline{G}}(B)}   \;=\; \frac{2(2r)m}{m(m-1+2r)} \;\leq \; \frac{4r}{m}\,.
	\end{equation*}
	
	For each $n$, let $G_n$ be the above graph with $r=r_n$ and $m=m_n$ chosen so that $n=2m_n+2r_n$ and $r_n \sim n^{1/4}$ as $n\rightarrow \infty$.
	Then $m_n\sim n/2$ and 
	\[    \max\{\Gap(G_n),\Gap(\overline{G}_n))\}    \;\leq \;   \max\left\{  \frac{2}{r_n^3},\frac{4r_n}{m_n}\right\}
	\;\sim\;      \max\left\{  \frac{2}{n^{3/4}},\frac{4n^{1/4}}{n/2}\right\}   \;=\;  \frac{8}{n^{3/4}}\,  \hspace{5mm}\hbox{as }n\rightarrow\infty.
	\]
	
	\section*{Acknowledgement}
	The authors are grateful to Ada Chan at York University for constructive conversations and to Mark Kempton and Adam Knudson at Brigham Young University for sharing their data and examples.
	
	\section*{Funding}
	Sooyeong Kim is supported in part by funding from the Fields Institute for Research in Mathematical Sciences and from the Natural Sciences and Engineering Research Council of Canada (NSERC). 
	Neal Madras is supported in part by NSERC Grant No.\ RGPIN-2020-06124.


\begin{thebibliography}{10}
	
	\bibitem{aksoy2018maximum}
	S.~G. Aksoy, F.~Chung, M.~Tait, and J.~Tobin.
	\newblock The maximum relaxation time of a random walk.
	\newblock {\em Advances in Applied Mathematics}, 101:1--14, 2018.
	
	\bibitem{aldous2002reversible}
	D.~Aldous and J.~Fill.
	\newblock Reversible {M}arkov chains and random walks on graphs.
	\newblock \url{http://www.stat.berkeley.edu/~aldous/RWG/book.html}, 2002.
	
	\bibitem{aouchiche2013survey}
	M.~Aouchiche and P.~Hansen.
	\newblock A survey of {N}ordhaus-{G}addum type relations.
	\newblock {\em Discrete Applied Mathematics}, 161(4--5):466--546, 2013.
	
	\bibitem{chung1997spectral}
	F.~Chung.
	\newblock {\em Spectral {G}raph {T}heory}, volume~92 of {\em CBMS Regional
		Conference Series in Mathematics}.
	\newblock American Mathematical Society, 1997.
	
	\bibitem{faught2023nordhaus}
	J.~N. Faught, M.~Kempton, and A.~Knudson.
	\newblock A {N}ordhaus-{G}addum type problem for the normalized {L}aplacian
	spectrum and graph {C}heeger constant.
	\newblock {\em arXiv preprint arXiv:2304.01979}, 2023.
	
	\bibitem{JeSi}
	M.~Jerrum and A.~Sinclair.
	\newblock Approximating the permanent.
	\newblock {\em SIAM Journal on Computing}, 18:1149--1178, 1989.
	
	\bibitem{kim2023bounds}
	S.~Kim, N.~Madras, A.~Chan, M.~Kempton, S.~Kirkland, and A.~Knudson.
	\newblock Bounds on {K}emeny's constant of a graph and the {N}ordhaus-{G}addum
	problem.
	\newblock {\em arXiv preprint arXiv:2309.05171}, 2023.
	
	\bibitem{LaSo}
	G.~F. Lawler and A.~D. Sokal.
	\newblock Bounds on the $l^2$ spectrum for {M}arkov chains and {M}arkov
	processes: {A} generalization of {C}heeger's inequality.
	\newblock {\em Transactions of the American Mathematical Society},
	309:557--580, 1988.
	
	\bibitem{LPW}
	D.~A. Levin, Y.~Peres, and E.~L. Wilmer.
	\newblock {\em Markov {C}hains and {M}ixing {T}imes}.
	\newblock American Mathematical Society, 2009.
	
	\bibitem{nordhaus1956complementary}
	E.~A. Nordhaus and J.~W. Gaddum.
	\newblock On complementary graphs.
	\newblock {\em American Mathematical Monthly}, 63(3):175--177, 1956.
	
	\bibitem{rosenthal2003}
	J.~S. Rosenthal.
	\newblock Asymptotic variance and convergence rates of nearly-periodic {MCMC}
	algorithms.
	\newblock {\em Journal of the American Statistical Association}, 98:169--177,
	2003.
	
\end{thebibliography}

\end{document}